\documentclass[11 pt]{amsart}
\usepackage{color}
\usepackage{verbatim}
\usepackage{amssymb}
\usepackage{amscd}
\usepackage{amsmath}
\usepackage{eucal}
\usepackage{setspace}
\usepackage{url}
\usepackage[utf8]{inputenc}

\allowdisplaybreaks
\newcommand{\Tr}{\mathrm{Tr}}

\renewcommand{\Im}{\mathrm{Im}}

\newcommand{\Sp}{\mathrm{Sp}}
\renewcommand{\H}{\mathbb H}
\newcommand{\T}[1]{{}^t{{#1}}}

\newcommand{\A}{{\mathbb A}}
\newcommand{\Q}{{\mathbb Q}}
\newcommand{\Z}{{\mathbb Z}}
\newcommand{\R}{{\mathbb R}}
\newcommand{\C}{{\mathbb C}}
\newcommand{\bs}{\backslash}

\newcommand{\GL}{{\rm GL}}

\newcommand{\SL}{{\rm SL}}

\newcommand{\bc}{{\rm bc}}

\newcommand{\GSp}{{\rm GSp}}

\newcommand{\sq}{\mathfrak{S}}

\DeclareMathOperator{\Cl}{Cl}
\newcommand{\mat}[4]{{\setlength{\arraycolsep}{0.5mm}\left(\begin{array}{cc}#1&#2\\#3&#4\end{array}\right)}}

\newcommand{\forget}[1]{}

\newtheorem{lemma}{Lemma}[section]
\newtheorem{theorem}{Theorem}

\newtheorem{proposition}[lemma]{Proposition}

\theoremstyle{remark}
\newtheorem{remark}[lemma]{Remark}

\begin{document}

\bibliographystyle{plain}

\title[Siegel cusp forms and Fourier coefficients]{Siegel cusp forms of degree 2 are determined by their fundamental Fourier coefficients}

\author{Abhishek Saha}
\address{ETH Z\"urich -- D-MATH\\
  R\"amistrasse 101\\
  8092 Z\"urich\\
  Switzerland} \email{abhishek.saha@math.ethz.ch}

\subjclass[2000]{11F30, 11F37, 11F46, 11F50}

 \begin{abstract}
 We prove that a Siegel cusp form of degree $2$ for the full modular group is determined by its set of Fourier coefficients $a(S)$ with $4\det(S)$ ranging over odd squarefree integers. As a key step to our result, we also prove that a classical cusp form of half-integral weight and level $4N$, with $N$ odd and squarefree, is determined by its set of Fourier coefficients $a(d)$ with $d$ ranging over odd squarefree integers, a result that was previously known only for Hecke eigenforms.
 \end{abstract}

 \maketitle

\section{Introduction}
Let $S_k(\Sp_4(\Z))$ denote the space of Siegel cusp forms of degree  (genus)  $2$ and weight $k$ with respect to the full modular group $\Sp_4(\Z)$. We refer the reader to Section~\ref{siegelsec} for the precise definitions of these objects. It is well known that any $F \in S_k(\Sp_4(\Z))$ has a Fourier expansion of the form
\begin{equation}\label{fouriercoeffsiegelform}
 F(Z)=\sum_Sa(F, S)e^{2\pi i{\rm Tr}(SZ)},\qquad Z\text{ in the Siegel upper half space}.
\end{equation}
Here, the Fourier coefficients $a(F, S)$ are indexed by symmetric, semi-integral, positive-definite $2\times2$ matrices $S$, i.e., matrices of the form
$$
 S=\mat{a}{b/2}{b/2}{c},\qquad a,b,c\in\Z, \qquad a>0, \qquad D = 4ac -b^2 >0.
$$
If $\gcd(a,b,c)=1$, then $S$ is called \emph{primitive}. If $-D$ is a fundamental discriminant, then $S$ is called \emph{fundamental}. Observe that if $S$ is fundamental, then it is automatically primitive. Observe also that if $D$ is odd, then $S$ is fundamental if and only if $D$ is squarefree.

Zagier \cite[p.\ 387]{zagier81} has shown that if $F \in S_k(\Sp_4(\Z))$ is such that $a(F, S) = 0$ for all primitive $S$, then $F = 0$. This result has recently been generalized by Yamana \cite{yamana09} to cusp forms with level and of higher degree. Related results have been obtained by Breulmann--Kohnen~\cite{breul-kohn}, Heim~\cite{heim-sep}, Katsurada~\cite{katsur-sep} and Scharlau--Walling~\cite{schar-wall}.

Zagier's result tells us that elements of $S_k(\Sp_4(\Z))$ are uniquely determined by their primitive Fourier coefficients. From the point of view of  representation theory, however, it is the smaller set of fundamental Fourier coefficients that are the more interesting objects. In particular, as explained in more detail below, it is a natural and important question whether elements of $S_k(\Sp_4(\Z))$ are uniquely determined by their fundamental coefficients. In this paper, we show that this is indeed the case.

\begin{theorem}\label{maintheorem}Let $F \in S_k(\Sp_4(\Z))$ be such that $a(F, S) = 0$ for all but finitely many matrices $S$ such that $4 \det(S)$ is odd and squarefree. Then $F = 0$.\end{theorem}

\begin{remark} Theorem~\ref{maintheorem} implies  that if $F \neq 0$ is an element of $S_k(\Sp_4(\Z))$, then $a(F,S) \neq 0$ for some fundamental $S$. In the case that $F$ is an eigenform for all the Hecke operators, this implies that the automorphic representation attached to $F$ has an unramified Bessel model; see~\cite[Lemma 5.1.1]{transfer}. As a result, Theorem~\ref{maintheorem} has several useful consequences. In particular, it removes the key assumption from results due to Furusawa~\cite{fur}, Pitale--Schmidt~\cite{pitsch} and the author~\cite{saha1, sah2} on the degree 8 $L$-function $L(s, F \times g)$ where $F\in S_k(\Sp_4(\Z))$ is a Hecke eigenform, and $g$ is a (classical) newform of integral weight. Furusawa~\cite{fur} discovered a remarkable integral representation for this $L$-function and used it to prove a special value result in the spirit of Deligne's conjecture. However, for his local calculations, he assumed the existence of a certain unramified Bessel model associated to $F$. To ensure that such a Bessel model exists, he made the assumption~\cite[(0.1)]{fur} that $F$ has a non-zero fundamental Fourier coefficient. Subsequent special value results in the spirit of Deligne's conjecture by Pitale--Schmidt~\cite{pitsch} and the author~\cite{ saha1, sah2} also assumed that $F$ satisfies this property. In forthcoming work of the author with Pitale and Schmidt~\cite{transfer}, a functorial transfer to $\GL(4)$ will be proved for Siegel cusp forms that satisfy this assumption. With Theorem~\ref{maintheorem}, we know now that the assumption~\cite[(0.1)]{fur} always holds, and thus all the special value and transfer theorems mentioned above hold unconditionally for eigenforms $F$ in $S_k(\Sp_4(\Z))$.
\end{remark}

Our method to prove Theorem~\ref{maintheorem} is quite different from those of the papers mentioned above. We exploit the Fourier-Jacobi expansion of $F$ and the relation between Jacobi forms and (elliptic) cusp forms of half-integral weight. In particular, we reduce Theorem~\ref{maintheorem} to an analogous result for the half-integral weight case which is of independent interest.

Let $S_{k+\frac{1}{2}}(N, \chi)$ denote the space of holomorphic cusp forms of weight $k+\frac{1}{2}$ and character $\chi$ for the group $\Gamma_0(N)$;
 we refer the reader to Section~\ref{halfintsec} for the details. We prove the following theorem.

\begin{theorem}\label{impprop}Let $N$ be a positive integer that is divisible by $4$ and $\chi : (\Z / N \Z)^\times \rightarrow \C^\times$ be a character. Write $\chi = \prod_{p | N} \chi_p$ and assume that the following conditions are satisfied:

\begin{enumerate}
\item $N$ is not divisible by $p^3$ for any prime $p$,
\item If $p$ is an \emph{odd} prime such that $p^2$ divides $N$, then $\chi_p \ne 1$.

\end{enumerate}
For some $k \ge 2$, let $f \in S_{k+\frac{1}{2}}(N, \chi)$ be such that $a(f, d) = 0$ for all but finitely many odd squarefree integers $d$. Then $f=0$.\end{theorem}

\begin{remark}Suppose that $N/4$ is odd and squarefree, and let $f_1, f_2$  belong to the Kohnen plus-space $S^+_{k+\frac12}(N)$ which is defined to be the space of holomorphic cusp forms $f$ of weight $k + \frac12$, level $N$ and trivial character satisfying $a(f,n) = 0$ if $n \equiv (-1)^{k+1}$ or $2 \mod{4}$. Then, if $a(f_i, d) \neq 0$ with $d$ odd and squarefree, we must have that $(-1)^k d$ is a fundamental discriminant. Thus, if $a(f_1, |D|) = a(f_2, |D|)$ for $f_1, f_2$ as above and all but finitely many fundamental discriminants $D$ with $(-1)^kD >0$, then Theorem~\ref{impprop} implies that $f_1 = f_2$. In other words, half-integral weight cusp forms of level $N$ (with $N/4$ odd and squarefree) in the Kohnen plus-space are uniquely determined by their Fourier coefficients at almost all fundamental discriminants. Previously this was known only in the case when $f_1, f_2$ are both Hecke eigenforms by work of Luo and Ramakrishnan~\cite[Thm. E]{luoram}.

 Our method is very different from that of Luo--Ramakrishnan. Indeed, Luo--Ramakrishnan used Waldspurger's formula to relate the squares of Fourier coefficients of half-integral weight newforms with twisted $L$-values of integral weight  newforms, and then appealed to a theorem proved in the same paper that integral weight newforms are uniquely determined by the central $L$-values of their twists with quadratic characters. In contrast, because the cusp forms in Theorem~\ref{impprop} are not necessarily Hecke eigenforms, the squares of their Fourier coefficients do not equal central $L$-values. Thus, there seems to be no obvious way to reduce our problem to averages of $L$-functions. Instead, we estimate averages of squares of Fourier coefficients of half-integral weight forms directly,  using a technical bound obtained from relations between the Hecke operators, an asymptotic formula due to Duke and Iwaniec, and a simple sieving procedure to go from the average over all coefficients to those that are squarefree and coprime to a specified set of primes. \end{remark}

  As mentioned earlier, Theorem~\ref{maintheorem} follows from Theorem~\ref{impprop} using the Fourier-Jacobi expansion of Siegel cusp forms, and the link between Jacobi forms and classical holomorphic forms of half integer weight. Actually, for the purpose of proving Theorem~\ref{maintheorem}, we only need Theorem~\ref{impprop} when $N= 4p$ or $4p^2$ where $p$ is  an odd prime.

We have already mentioned the connection between Theorem~\ref{maintheorem} as applied to a Hecke eigenform $F$ and the existence of an unramified Bessel model for the corresponding automorphic representation. It turns out that Theorem~\ref{maintheorem}  also has an interesting connection with the central $L$-values of certain twisted $L$-functions associated with $F$.  Let us briefly explain this connection. Let $D>0$ be a positive integer such that $-D$ is a fundamental discriminant. Let $\Cl_D$ denote the ideal class group of the field $\Q(\sqrt{-D})$. There is a well-known natural isomorphism between $\Cl_D$ and the $\SL(2,\Z)$-equivalence classes of primitive semi-integral matrices with determinant equal to $D/4$.

Now let $F \in  S_k(\Sp_4(\Z))$ be a non-zero eigenfunction for all the Hecke operators and let $a(F,S)$ denote its Fourier coefficients as in~\eqref{fouriercoeffsiegelform}. For a character $\Lambda$ of $\Cl_D$, let \begin{equation}\label{eq:alf}
 a(D, \Lambda;F) =\sum_{c\in \Cl_D}\overline{\Lambda(c)}a(F,c),
\end{equation}
a quantity which is well-defined in view of the invariance of Fourier
coefficients under $\SL(2,\Z)$.

 In~\cite{boch-conj}, B\"ocherer made a
remarkable conjecture that connects the Fourier coefficients of $F$ with the central $L$-values of the twists of $F$ with quadratic characters. More refined and generalized versions of this conjecture were proposed in works of Martin, Furusawa and Shalika, see~\cite{furusawa-martin} for instance. Global evidence for this conjecture was provided in recent work of the author with Kowalski and Tsimerman~\cite{kst2}. The  question of vanishing of  $a(D, \Lambda;F)$ is also closely related to the global Gross-Prasad conjecture
  for $(SO(5), SO(2))$.

In particular, suppose that $F$ is not a Saito-Kurokawa lift. Then Theorem~\ref{maintheorem} implies that we can find $D$, $\Lambda$ as above, with $\Lambda$ non-trivial, such that $a(D, \Lambda;F) \neq 0$. Now, the conjectures referred to above imply that $L(\frac{1}{2}, \pi_F \times
\theta(\Lambda^{-1})) \neq 0 $ where $\pi_F$ is the cuspidal automorphic representation on $\GSp_4(\A)$ associated to $F$ and $\theta(\Lambda^{-1})$ is the automorphic induction of $\Lambda^{-1}$ to $\GL_2(\A)$.

Therefore, \emph{assuming those conjectures}, Theorem~\ref{maintheorem} tells us that for any non Saito-Kurokawa Hecke eigenform $F \in  S_k(\Sp_4(\Z))$, there exists an imaginary quadratic field $K=\Q(\sqrt{-D})$ and a non-trivial character $\Lambda$ of the ideal class group $\Cl_D$ of $K$, such that $L(\frac{1}{2}, \pi_F \times
\theta(\Lambda^{-1})) \neq 0$. This can be viewed as a $\GSp(4)$-analogue of results due to Rohrlich~\cite{rohrlich} and Waldspurger~\cite{waldscentral} on non-vanishing of central $L$-values for $\GL(2)$. It also has an intriguing connection — via the work of Ash and Ginzburg~\cite{ashginz} — with a certain $p$-adic $L$-function associated to
$\bc_{K/Q}\Pi(\pi_F )$; here $\Pi(\pi_F)$ denotes the functorial transfer of $\pi_F$ to $\GL_4(\A)$, the
existence of which will be proved in [13], while $\bc_{K/Q}\Pi(\pi_F )$ denotes the base-change of $\Pi(\pi_F )$ to $\GL_4(\A_K)$.

Finally, regarding possible generalizations of Theorem~\ref{maintheorem}, we refer the reader to the remarks at the end of Subsection~\ref{s:proofmaint}.
\section*{Acknowledgements}I would like to express my thanks to Emmanuel Kowalski for his help with several parts of this paper and for his feedback on the draft version.  I would also like to thank {\"O}zlem Imamoglu, Paul Nelson, Ameya Pitale, Ralf Schmidt and Ramin Takloo-Bighash for discussions and feedback. Finally, thanks are due to the referee for some suggestions which improved this paper.

\section{Siegel cusp forms of degree 2}
In this section, we will prove Theorem~\ref{maintheorem} assuming the truth of Theorem~\ref{impprop}.
\subsection{Notations}\label{siegelsec}For any commutative ring $R$ and positive integer $n$, $M_n(R)$
  denotes the ring of $n$ by $n$ matrices with entries in $R$ and $\GL_n(R)$ denotes the group of invertible matrices.  If
  $A\in M_n(R)$, we let $\T{A}$ denote its transpose. We say that a symmetric matrix in $M_n(\Z)$ is semi-integral if
  it has integral diagonal entries and half-integral off-diagonal
  ones.  Denote by $J$ the $4$ by $4$ matrix given by
$$
J =
\begin{pmatrix}
0 & I_2\\
-I_2 & 0\\
\end{pmatrix}.
$$ where $I_2$ is the identity matrix of size 2.

Define the algebraic group
  $\Sp_4$ over $\Z$ by
$$
\Sp_4(R) = \{g \in \GL_4(R) \; | \; \T{g}Jg =
  J\}
$$
for any commutative ring $R$.

The Siegel upper-half space of degree 2 is defined by
$$
\H_2 = \{ Z \in M_2(\C)\;|\;Z =\T{Z},\ \Im(Z)
  \text{ is positive definite}\}.
$$
We define
$$
 g \langle Z\rangle = (AZ+B)(CZ+D)^{-1}\qquad\text{for }
 g=\begin{pmatrix} A&B\\ C&D \end{pmatrix} \in \Sp_4(\R),\;Z\in \H_2.
$$
We let $J(g,Z) = CZ + D$ and use $i_2$ to denote the point $\begin{pmatrix}i&\\& i \end{pmatrix} \in \H_2$.

Let $\Gamma$ denote the group $\Sp_4(\Z)$. We say that $F \in S_k(\Gamma)$ if $F$ is a holomorphic function on
$\H_2$ which satisfies
$$
F(\gamma \langle Z\rangle) = \det(J(\gamma,Z))^k F(Z)
$$
for $\gamma \in \Gamma$, $Z \in \H_2$, and vanishes at the
cusps. Such an $F$ is called a Siegel cusp form of degree (genus) 2, weight $k$ and full level.

\subsection{The Fourier--Jacobi expansion} It is well-known that $F\in S_k(\Gamma)$ has a Fourier expansion $$F(Z)
=\sum_{T > 0} a(F, T) e(\Tr(TZ)),$$ where $T$ runs through all symmetric, semi-integral, positive-definite matrices of
size two.

This immediately shows that any $F \in S_k(\Gamma)$ has a ``Fourier--Jacobi expansion"

\begin{equation}\label{fjexpand}F(Z) = \sum_{m > 0} \phi_m(\tau, z) e(m \tau')\end{equation} where we write $Z= \begin{pmatrix}\tau&z\\z&\tau' \end{pmatrix}$ and for each $m>0$,
\begin{equation}\label{jacobifourier}\phi_m(\tau, z) = \sum_{\substack{n,r \in \Z \\ 4nm> r^2}}a \left(F, \mat{n}{r/2}{r/2}{m}\right) e(n \tau) e(r z) \in J_{k,m}^{\text{cusp}}. \end{equation} Here $J_{k,m}^{\text{cusp}}$ denotes the space of Jacobi cusp forms of weight $k$ and index $m$; for details see~\cite{eichzag}.

If we put $c(n,r) = a \left(F, \mat{n}{r/2}{r/2}{m}\right)$, then~\eqref{jacobifourier} becomes

$$\phi_m(\tau, z) =  \sum_{\substack{n,r \in \Z \\ 4nm> r^2}} c(n,r) e(n \tau) e(r z),$$ and this is called the Fourier expansion of the Jacobi form $\phi_m$.

\begin{lemma}\label{iwanequiv} Let $S = \mat{n}{r/2}{r/2}{m}$ with $\gcd(m,r,n) = 1$ and $4nm> r^2$. Then there exists a matrix $S' = \mat{n'}{r'/2}{r'/2}{m'}$ such that $S' = \T{A}SA$ for some $A \in \SL_2(\Z)$ and $m'$ is equal to an odd prime.
\end{lemma}
\begin{proof}  It is a classical result (going back at least to Weber~\cite{weber82}) that the quadratic form $mx^2 + rxy + ny^2$ represents infinitely many primes. So, let $x_0$, $y_0$ such that $mx_0^2 + rx_0y_0 + ny_0^2$ is an odd prime. Since this implies $\gcd(x_0, y_0) = 1$, we can find integers $x_1$, $y_1$ such that $A= \mat{y_1}{y_0}{x_1}{x_0} \in \SL_2(\Z).$ Then $S' = \T{A}SA$ has the desired property.
\end{proof}

Now we can prove the following proposition.

\begin{proposition}\label{propfj}Let $F \in S_k(\Gamma)$ be non-zero with a Fourier--Jacobi expansion as in~\eqref{fjexpand}. Then there exists an odd prime $p$ such that $\phi_p \ne 0$.
\end{proposition}
\begin{proof}By~\cite[p. 387]{zagier81} and the fact that $F$ is a cusp form, we know $a(F,S) \neq 0$ for some $S = \mat{n}{r/2}{r/2}{m}$ with $\gcd(m,r,n) = 1$ and $4nm> r^2$. Since $a(F,S)$ only depends on the $\SL_2(\Z)$-equivalence class of $S$, it follows by Lemma~\ref{iwanequiv} that $a(F,S') \neq 0$ for $S' = \mat{n'}{r'/2}{r'/2}{m'}$ where $m'=p$ is an odd prime. The result follows immediately.
\end{proof}

\subsection{Proof of Theorem~\ref{maintheorem}}\label{s:proofmaint} Let $F \in S_k(\Gamma)$ be non-zero. We need to show that $a(F,S) \neq 0$ for infinitely many matrices $S$ such that $4 \det(S)$ is odd and squarefree. Let $F$ have a Fourier--Jacobi expansion as in~\eqref{fjexpand}. By Proposition~\ref{propfj}, we can find an odd prime $p$ such that the Fourier expansion of $\phi_p$ given by $$\phi_p(\tau, z) = \sum_{\substack{n,r \in \Z \\ 4np> r^2}} c(n,r) e(n \tau) e(r z)$$ has the property that $ c(n,r) \neq 0$ for some $n,r$.

We first deal with the case when $k$ is even. For $N \ge 1$, put $$a(N) = \sum_{\substack{0 \le \mu \le 2p-1 \\ \mu^2 \equiv -N \pmod{4p}}} c\left(\frac{N+\mu^2}{4p}, \mu \right) $$ and let $$h(\tau) = \sum_{N=1}^\infty a(N) e(N \tau).$$

Since $\phi_p \ne 0$, we have~\cite[Thm. 5.6]{eichzag} that $h(\tau)$ is a non-zero element of $S_{k-\frac{1}{2}}(4p)$ (see Section~\ref{halfintsec} for notations and definitions relating to half-integral weight forms). So by Theorem~\ref{impprop}, we conclude that $a(D)$  is not equal to zero for infinitely many $D$ that are odd and squarefree. For any of these $D$, there exists a $\mu$ such that  $c\left(\frac{D+\mu^2}{4p},\mu \right) = a \left(F, \mat{\frac{D+\mu^2}{4p}}{\mu/2}{\mu/2}{p}\right)$ is not equal to zero. This completes the proof for the even weight case.

Next, let $k$ be odd. In this case, if we try to define $h(\tau)$ as above it would be automatically equal to zero, so we need to twist things a bit. Let $\chi$ be any \emph{odd} Dirichlet character $\ \bmod \ p$, i.e., $\chi(-1) = -1$. Define $$a(N) =  \sum_{\substack{0 \le \mu \le 2p-1 \\ \mu^2 \equiv -N \pmod{4p}}} \chi(\mu) c\left(\frac{N+\mu^2}{4p}, \mu \right)  $$ and let $$h(\tau) = \sum_{N=1}^\infty a(N) e(N \tau).$$ Let $\epsilon_4$ be the unique Dirichlet character of conductor 4. In this case, we have~\cite[p. 70]{eichzag} that $$h(\tau) \in S_{k-\frac{1}{2}}(4p^2, \epsilon_4\chi ),$$ and $h(\tau) \ne 0$. The rest of the proof is identical to the case $k$ even above.

\begin{remark} For a detailed description of the correspondence between Jacobi modular/cusp forms of index $m$ and half-integer weight modular/cusp forms (with level depending on $m$), see~\cite[Abschnitt 4]{skoruppathesis}.
\end{remark}

\begin{remark} As is clear from the argument, we implicitly prove the following stronger result.

Let $F \in S_k(\Sp_4(\Z))$ be non-zero. Then there are infinitely many primes $p$ with the property that there are infinitely many matrices $T$ satisfying all the following: (a) $a(F, T) \neq 0$, (b) $-4 \det(S)$ is an odd fundamental discriminant, and (c) $T$ represents $p$ integrally.

\end{remark}

\begin{remark}Our method actually gives a quantitative lower bound on the number of non-vanishing fundamental Fourier coefficients; furthermore, it also works with some modifications for Siegel cusp forms with respect to the Siegel congruence subgroup $\Gamma_0^{(2)}(N)$ under some hypotheses. For the details of these extensions --- along with an application to the problem of simultaneous non-vanishing of $L$-functions --- we refer the reader to a forthcoming joint paper paper~\cite{sahaschmidt} with Ralf Schmidt.
\end{remark}

\begin{remark}It is an interesting open question whether a version of Theorem~\ref{maintheorem} is true for vector-valued Siegel cusp forms of degree 2. The main difficulty is the translation into a problem about Jacobi forms and then half-integral weight forms. One may also try to generalize Theorem~\ref{maintheorem} to the case of non-cuspidal forms (e.g. Eisenstein series). The main issue there, as pointed by the referee, is that it is difficult to distinguish the growth of Fourier coefficients of Siegel- and Klingen-Eisenstein series.
\end{remark}

\begin{remark}In a more general context, one can ask whether a cusp form of degree $n$ is determined by its Fourier coefficients $A(T)$ where $T$ corresponds to
a maximal lattice. As pointed out by the referee, this question, while beyond the methods of this paper, seems to the be the appropriate generalization.

\end{remark}

\section{Modular forms of half-integral weight}\label{halfintmainsec}
The object of this section is to prove Theorem~\ref{impprop}.
\subsection{Notations and preliminaries}\label{halfintsec}
The group $\SL_2(\R)$ acts on the upper half-plane $\H$ by $$\gamma z = \frac{az+b}{cz+d},$$ where $\gamma = \mat{a}{b}{c}{d} $ and $z=x+iy $. For a positive integer $N$, let $\Gamma_0(N)$ denote the congruence subgroup consisting of matrices $\mat{a}{b}{c}{d}$ in $\SL_2(\Z)$ such that $N$ divides $c$. For a complex number $z$, let $e(z)$ denote $e^{2\pi i z}$.

Let $\theta(z) = \sum_{n = -\infty}^\infty e(n^2 z)$ be the standard theta function on $\H$. If $A = \mat{a}{b}{c}{d} \in \Gamma_0(4)$, we have $\theta(Az) = j(A, z)\theta(z),$ where $j(A, z)$ is the so called $\theta$-multiplier. For an explicit formula for $j(A, z)$, see~\cite{Shimura3} or~\cite{Serre-Stark}.

Let $k \ge 2$ be a fixed integer. For a positive integer $N$ divisible by $4$ and a Dirichlet character $\chi\mod N$, let $S_{k+\frac{1}{2}}(N, \chi)$ denote the space of holomorphic cusp forms of weight $k+\frac{1}{2}$ and character $\chi$ for the group $\Gamma_0(N)$. In other words, a function $f : \H \rightarrow \C$ belongs to $S_{k+\frac{1}{2}}(N, \chi)$ if
\begin{enumerate}

\item $f(Az) = \chi(d) j(A, z)^{2k +1} f(z)$ for every $A = \mat{a}{b}{c}{d} \in \Gamma_0(N)$,
\item $f$ is holomorphic,
\item $f$ vanishes at the cusps.

\end{enumerate}

It is easy to check that the space $S_{k+\frac{1}{2}}(N, \chi)$ consists only of $0$ unless $\chi(-1)=1$.
We will use $S_{k+\frac{1}{2}}(N)$ to denote $S_{k+\frac{1}{2}}(N, 1)$ where $1$ denotes the trivial character.

If $\chi$ is a Dirichlet character such that the conductor of $\chi$ divides $N$, then there is a unique Dirichlet character $\chi' \bmod \ N$ attached to $\chi$. By standard abuse of notation, we will use $S_{k+\frac{1}{2}}(N, \chi)$ to denote $S_{k+\frac{1}{2}}(N, \chi')$.

Any $f \in  S_{k+\frac{1}{2}}(N, \chi)$ has the Fourier expansion

$$f(z) = \sum_{n > 0} a(f, n)e(nz).$$ We let $\tilde{a}(f,n)$ denote the ``normalized" Fourier coefficients, defined by
 $$\tilde{a}(f,n) = a(f,n)n^{ \frac14-\frac k2}.$$ For convenience, we define $$S_{k+\frac12}(*) = \bigcup_{ 4|N} \bigcup_{\chi\bmod N } S_{k+\frac12}(N, \chi).$$

For $f, g \in S_{k+\frac12}(*)$, we define the Petersson inner product $\langle f, g\rangle$ by $$\langle f, g\rangle = [\Gamma_0(4): \Gamma_0(N)]^{-1} \int_{\Gamma_0(N) \bs \H}f(z) \overline{g(z)} y^{k + \frac12} \frac{dx dy}{y^2},$$ where we choose an integer $N$ that is a multiple of $4$ and a character $\chi \bmod N$ such that both $f,g \in S_{k+\frac12}(N, \chi)$; if such a $(N, \chi)$ does not exist, we put $\langle f, g\rangle = 0.$

\subsection{Two useful lemmas}

In this subsection, we state two useful lemmas due to Serre and Stark~\cite{Serre-Stark}. These lemmas will be used frequently in the sequel.

\begin{lemma}[Lemma 4 of~\cite{Serre-Stark}]\label{lema4}Suppose that $4$ divides $N$ and $$f(z) = \sum_{n>0}a(f,n)e(nz)$$ belongs to $S_{k+\frac12}(N, \chi).$ Then $$g(z) = \sum_{\substack{n>0\\(n,p)=1}}a(f,n)e(nz)$$ belongs to $S_{k+\frac12}(Np^2, \chi).$

\end{lemma}
\medskip
\begin{lemma}[Lemma 7 of~\cite{Serre-Stark}]\label{lema7}Let $p$ be a prime and $\epsilon_{p}$ be the quadratic character associated to the field $\Q(\sqrt{p})$. Let $N$ be an integer divisible by 4 and  $f(z)= \sum_{n>0}a(f,n)e(nz)$ be a non-zero element of $ S_{k+\frac12}(N, \chi)$ such that $a(f,n) = 0$ for all $n$ not divisible by $p$. Then $p$ divides $N/4$, the conductor of $\chi \epsilon_p$ divides $N/4$, and $g(z) = f(z/p)$ belongs to $ S_{k+\frac12}(N/p, \chi \epsilon_p)$.

\end{lemma}

\subsection{Some estimates on averages of squares of Fourier coefficients} Throughout the rest of Section~\ref{halfintmainsec}, we fix an integer $k \ge 2$.

The following theorem of Duke and Iwaniec~\cite{dukeiwan} will be important for us. Note that the form in which we state the result here differs slightly from their presentation; this is due to making explicit constants implicit, etc.

\begin{theorem}[Duke--Iwaniec]\label{dukeiwansharp}Let $N$ be a positive integer that is divisible by $4$ and $\chi$ a Dirichlet character $\ \bmod \ N$. Then, for any $f \in S_{k + \frac12}(N, \chi)$, $f \ne 0$, there exists a positive squarefree integer $N_f$ divisible by all primes dividing $N$ and a positive real number $E_f >0$ such that the following holds.

 For $r$ any squarefree integer coprime to $N_f$, we have $$\sum_{n > 0} |\tilde{a}(f, nr)|^2 e^{-n/X} = E_f X + O(X^{\frac12 + \epsilon})$$ where the implied constant depends on $f$, $r$ and $\epsilon$.
\end{theorem}
\begin{proof} Note that we can write $f$ as a linear combination of the Poincare series $P_m$ of level $N$, character $\chi$ and weight $k + \frac12$; here $m$ varies over the positive integers. Suppose we have $$f = \sum_{i=1}^b P_{m_i}.$$ Let $N(f)$ be the set of primes that divide $N$ or divide $m_i$ for some $1 \le i \le b$. We define $N_f$ to be the product of all the primes in $N(f)$. The result now follows from~\cite[Lemmas 4 and 5]{dukeiwan}.
\end{proof}
 \begin{remark}An extension of the Duke--Iwaniec method to the case when $r$ is no longer squarefree seems to involve certain technical dificulties (but if accomplished would lead to a quicker proof of Theorem~\ref{impprop}). We will not attempt this extension here, but will instead use a different argument involving Hecke operators to bypass it; see Lemma~\ref{ur2lemma} below.
 \end{remark}

\begin{proposition}\label{lowerbound}Let $N$ be a positive integer that is divisible by $4$ and $\chi$ a Dirichlet character $\ \bmod \ N$. Let $f \in S_{k + \frac12}(N, \chi)$ and let $N_f$ be as in Theorem~\ref{dukeiwansharp}.  Suppose that $a(f,n)$ equals 0 whenever $n$ and $N$ have a common prime factor. Then there exists a positive real number $D_f>0$ with the following property.

For any squarefree integer $M$  divisible by $N_f$, there exists an integer $X_{f, M}$  such that $$\sum_{\substack{n > 0\\ (n, M) = 1}} |\tilde{a}(f, n)|^2 e^{-n/X}> \frac{\phi(M)}{M}D_f X$$ for all $X \ge X_{f, M}$.
 \end{proposition}
\begin{proof}Let $$g(z) = \sum_{\substack {n > 0\\ (n, N_f) = 1}}a(f, n) e(nz).$$ By Lemma~\ref{lema4}, we know that $g \in S_{k + \frac12}(NN_f^2, \chi).$ We next show that $g \neq 0$. Indeed, let $p_1, p_2,\ldots,p_t$ be the distinct primes dividing $N_f/(N, N_f)$. Let $g_0=f$, and for $1 \le i \le t$, define $$ g_i(z) = \sum_{\substack{ n > 0\\ (n, p_i) = 1}}a(g_{i-1}, n) e(nz).$$
Then by the assumption of the proposition and Lemma~\ref{lema4}, we have $g_0 =f$, $g_t = g$ and for each  $1 \le u \le t$, $g_u \in S_{k + \frac12}(N\cdot(p_1p_2\ldots p_{u})^2, \chi).$
Suppose for some $1 \le u \le t$, we have $g_u = 0$, $g_{u-1} \neq 0$. Then $a(g_{u-1}, n) = 0$ unless $p_u$ divides $n$. By Lemma~\ref{lema7}, this implies that $p_u$ divides $N(p_1p_2\ldots p_{u-1})^2/4$, a contradiction.

Thus we have $g_u \neq 0$ for $1 \le u \le t$. In particular $g \neq 0$. Let $M_0=M/N_f$. So $M_0$ is a squarefree integer coprime to $N_f$. Let $\mu(n)$ denote the Mobius function.  Using Theorem~\ref{dukeiwansharp}, we have \begin{align*}\sum_{\substack{n >0 \\ (n, M) = 1}} |\tilde{a}(f, n)|^2 e^{-\frac{n}X} &= \sum_{\substack{n >0 \\ (n, M_0) = 1}} |\tilde{a}(g, n)|^2e^{-\frac{n}X}\\&= \sum_{d|M_0}\mu(d)\sum_{n \ge 1} |\tilde{a}(g, nd)|^2e^{-\frac{nd}X}   \\&=\sum_{d|M_0}\frac{\mu(d)}{d}E_g X + o(X)  \\ &= \frac{\phi(M_0)}{M_0}E_g X + o(X). \end{align*} Note that $E_g >0$ is a constant depending on $g$ (or equivalently, on $f$). Now we can select $D_f$ to be any positive quantity smaller than $E_g \frac{N_f}{\phi(N_f)}$; the result follows immediately.
\end{proof}

The next proposition complements the last one by giving an uniform upper bound on the average of squares of Fourier coefficients.

\begin{proposition}\label{upperbound}Let $N$ be a positive integer that is divisible by $4$ and $\chi$ a Dirichlet character $\ \bmod \ N$. There exists a constant $C_{N, \chi}$, depending only on $N$ and $\chi$,  such that for all $f \in S_{k + \frac12}(N, \chi)$, we have $$\sum_{n > 0} |\tilde{a}(f, n)|^2 e^{-n/X} < C_{N,\chi} \langle f, f\rangle X$$ for all $X >0$.
\end{proposition}
\begin{proof}Let $f_1,\ldots f_j$ be an orthonormal basis of $S_{k + \frac12}(N, \chi)$. By Theorem~\ref{dukeiwansharp}, there exist constants $C_i$ for $1 \le i \le j$ such that  $$\sum_{n >0} |\tilde{a}(f_i, n)|^2 e^{-n/X}< C _iX$$ for all $X > 0$. Put $C_{N, \chi}= C_1 + C_2 + \ldots C_j$. By the Cauchy-Schwartz inequality, it follows that $$\sum_{n > 0} |\tilde{a}(f, n)|^2 e^{-n/X} < C_{N, \chi} \langle f, f\rangle X$$ for all $X >0$.
\end{proof}

Note that the above bound works for all $X$ and not just sufficiently large ones. We wrap up this subsection with the following elementary lemma that will be used later.

\begin{lemma}\label{easylemma} Let $y$ be any positive real number. We have $$\frac{-1 + \prod_{p >T, \ p \text{ prime }}\left(1 + \frac{y}{p^2} \right) }{\prod_{p \le T, \ p \text{ prime }}\left(1 - \frac{1}{p} \right)} \rightarrow 0$$ as $T \rightarrow \infty$.
\end{lemma}
\begin{proof}The numerator is bounded above by $-1 + e^{\sum_{p>T}\frac{y}{p^2}}$, which in turn is bounded above by $-1 + e^{\frac{y}{T}}$, which is asymptotically $\frac{y}{T}$. As for the denominator, $$\prod_{p \le T, \ p \text{ prime }}\left(1 - \frac{1}{p} \right) \gg_\epsilon \frac{1}{T^\epsilon}$$ for any $\epsilon>0$. The result follows.
\end{proof}

\subsection{Hecke operators and resulting bounds}
We begin by recalling a few operators that act on spaces of half-integral weight forms. First, for all primes $p$ coprime to $N$ there exist Hecke operators $T(p^2)$ acting on the space $S_{k+\frac12}(N, \chi)$; see~\cite[Thm. 1.7]{Shimura3}.

Next, for any odd squarefree integer $r$, let $U(r^2)$ be the operator on  $S_{k+\frac12}(N, \chi)$ defined by \begin{equation}\label{defur2}U(r^2)f = r^{\frac12 - k}\sum_{n >0 }a(f,r^2n)e(nz).\end{equation} It is known~\cite[Prop. 1.5]{Shimura3} that $U(r^2)$ takes $S_{k + \frac12}(N, \chi)$ to $S_{k + \frac12}(Nr/(N,r), \chi)$. Note also that we have\begin{equation}\label{defur3}\tilde{a}(U(r^2)f, n) = \tilde{a}(f, r^2 n).\end{equation}

The point of this subsection is to prove the following proposition.

\begin{proposition}\label{keyprop}Let $N$ be a positive integer that is divisible by $4$ and $\chi$ a Dirichlet character $\ \bmod \ N$. Let $f \in S_{k + \frac12}(N, \chi)$, $f \ne 0$, and suppose that $a(f,n) = 0$ unless $(n,N) = 1$. Then there are infinitely many odd squarefree integers $d$ such that $a(f,d) \neq 0$.

\end{proposition}

Let $\omega(r)$ denotes the number of distinct primes dividing $r$. We need the next two Lemmas for the proof of Proposition~\ref{keyprop}.

\begin{lemma}\label{ur2lemma}Let $h \in S_{k + \frac12}(N, \chi)$ be an eigenfunction for the Hecke operators $T(p^2)$ for all $p \nmid N$. Suppose that $\langle h, h\rangle = 1$.  Let $r$ be a squarefree integer coprime to $N$. Then for all $Y > 0$, we have $$\sum_{n > 0}|\tilde{a}(U(r^2)h , n)|^2 e^{-n/Y}\le 19^{\omega(r)} C_{N, \chi}Y$$ where $C_{N, \chi}$ is as in Proposition~\ref{upperbound}.
\end{lemma}
\begin{proof}Set $t = \omega(r)$. We use induction on $t$. The case $t=0$ (i.e., $r=1$) is just Proposition~\ref{upperbound}. Now let $t \ge 1$ and $r = p_1p_2\ldots p_t$. Put $r_1 = p_1p_2\ldots p_{t-1}$ and $g= U(r_1^2) h$. Then $g \in S_{k + \frac12}(Nr_1, \chi)$ is an eigenfunction of $T(p_t^2)$. This is because the operators $T(p_t^2)$ and $U(r_1^2)$ commute. By~\cite[Thm. 1.7]{Shimura3}, we have \begin{equation}\label{hecke}\tilde{a}(U( p_t^2)g, n) = \tilde{a}(g , n)(\lambda_{p_t^2} + \frac{\epsilon_{p_t}}{\sqrt{p_t}}) - \chi(p_t)^2\tilde{a}(g , n/p_t^2),\end{equation}where $\lambda_{p_t^2} \in [-2,2]$ is the normalized Hecke eigenvalue (here we are using the Shimura correspondence combined with the  Deligne bound) for $T(p_t^2)$ acting on $g$ and $\epsilon_{p_t}$ is a complex number such that $| \epsilon_{p_t} |\in \{ 0, 1 \}$. It follows that

\begin{align*}\sum_{n > 0}|\tilde{a}(U(r^2)h , n)|^2 e^{-\frac{n}{Y}} &=\sum_{n >0}|\tilde{a}(U( p_t^2)g , n)|^2e^{-\frac{n}{Y}}\\ &  \le 2\sum_{n>0} \left( |\tilde{a}(g , n)|^2 \big(\lambda_{p_t^2} + \frac{\epsilon_{p_t}}{\sqrt{p_t}}\big)^2 + |\tilde{a}(g , n/p_t^2)|^2\right)e^{-\frac{n}{Y}} \\& \leq 18 \sum_{n > 0}|\tilde{a}(g , n)|^2 e^{-\frac{n}{Y}}+ 2 \sum_{n> 0}|\tilde{a}(g , n)|^2 e^{-\frac{np_t^2}{Y}}\end{align*} where the penultimate step follows from~\eqref{hecke} and the Cauchy-Schwartz inequality, while the last step follows from the bound $(\lambda_{p_t^2} + \frac{\epsilon}{\sqrt{p_t}})^2 \le 9.$ Using the induction hypothesis, it follows that $$\sum_{n > 0}|\tilde{a}(U(r^2)h , n)|^2 e^{-n/Y} \le (19)^{t-1} C_{N, \chi}Y (18 + 2/p_t^2) \le  (19)^{t} C_{N, \chi}Y.$$ The lemma is proved.
\end{proof}

\begin{lemma}\label{lemmaheckebd} Let $f \in S_{k + \frac12}(N, \chi)$. There exists a constant $B_f$ such that for any positive squarefree integer $r$ coprime to $N$ and any $Y > 0$, we have $$\sum_{n > 0}|\tilde{a}(U(r^2) f , n)|^2 e^{-n/Y}\le 19^{\omega(r)} B_fY.$$
\end{lemma}
\begin{proof} Let $f_1,f_2, \ldots f_j$ be an orthonormal basis of $S_{k + \frac12}(N, \chi)$ such that $f_i$ is an eigenfunction for the Hecke operators $T(p^2)$ for all $p \nmid N$. We put $B_f = j \langle f, f\rangle C_{N, \chi}$. Now the proof is immediate from Lemma~\ref{ur2lemma} by writing $f$ as a linear combination of $f_i$ and using the Cauchy-Schwartz inequality.

\end{proof}

Now we can prove Proposition~\ref{keyprop}.

\begin{proof}Let $\sq$ denote the set of odd squarefree positive integers. Let $N_f$ be as in Theorem~\ref{dukeiwansharp}. For any positive squarefree integer $M$  divisible by $N_f$, define $$S(M,X;f) = \sum_{\substack{d\in \sq \\ (d, M) = 1}} |\tilde{a}(f, d)|^2 e^{-d/X}.$$

The proposition will be proved if we can show that for some suitable choice of $M$, $S(M,X;f) \rightarrow \infty$ as $X\rightarrow \infty$.

Using the identity $$\sum_{r^2 | n} \mu(r) = \begin{cases}1 & \text{ if } n \text{ is squarefree } \\ 0 & \text{ otherwise }\end{cases},$$ and~\eqref{defur3}, we can sieve to squarefree terms and get

$$S(M,X;f)  = \sum_{\substack{\ r\in \sq \\ (r, M) = 1}} \mu(r) \sum_{\substack{n >0 \\ (n, M) = 1}} |\tilde{a}(U(r^2)f, n)|^2e^{-r^2n/X} .$$

Thus, for $X \ge X_{f, M}$, we have, using Proposition~\ref{lowerbound},
\begin{equation}\label{keyineq} S(M,X;f)  >  \frac{\phi(M)}{M}D_f X - \sum_{\substack{r \ge 2, r\in \sq \\ (r, M) = 1}} \sum_{\substack{n >0 \\ (n, M) = 1}} |\tilde{a}(U(r^2)f, n)|^2e^{-r^2n/X} \end{equation}

Using Lemma~\ref{lemmaheckebd}, we get from~\eqref{keyineq} that $$S(M,X;f) >  \frac{\phi(M)}{M}D_f X - \sum_{\substack{r \ge 2, r \in \sq \\ (r, M) = 1}} 19^{\omega(r)}B_f \frac{X}{r^2}$$ for $X > X_{f, M}$. Now, we finally make a choice for $M$. Indeed we choose $$M = \prod_{p\le T, \ p \text{ prime }} p$$ to be a squarefree integer such that $M$ is divisible by $N_f$ and such that the following inequality holds $$D_f \prod_{p \le T} (1- \frac1p) > B_f \left( -1 + \prod_{p > T}(1 + \frac{19}{p^2}) \right).$$ The existence of such a $T$ is guaranteed by Lemma~\ref{easylemma}. With this choice of $M$, we see that $S(M,X;f) > A_f X$ for all sufficiently  large $X$ where $A_f$ is a positive number. This shows that $S(M,X;f) \rightarrow \infty$ as $X\rightarrow \infty$ and completes the proof of Proposition~\ref{keyprop}.

\end{proof}

\subsection{Proof of Theorem~\ref{impprop}}

Recall that $\chi = \prod_{p | N} \chi_p$ is a character whose conductor divides $N$ and such that the two conditions of Theorem~\ref{impprop} are satisfied. Let $f \in S_{k+ \frac12}(N, \chi)$ be non-zero. We need to show that there are infinitely many odd squarefree integers $d$ such that $a(f,d) \neq 0$.

Let $2=p_1, p_2,\ldots,p_t$ be the distinct primes dividing $N$. For $1 \le i \le t$, let $S_i = \{p_1, \ldots, p_i\}$. Put $g_0 = f$. We will construct a sequence of forms $g_i$, $1 \le i \le t$, such that \begin{enumerate}

\item $g_i \ne 0,$
\item $g_i \in S_{k+ \frac12}(N N_i, \chi \chi_i)$ where $N_i$ is composed of primes in $S_i$ and $\chi_i$ is a Dirichlet character whose conductor is not divisible by any prime outside $S_i$,
\item $a(g_i, n) = 0$ whenever $n$ is divisible by a prime in $S_i$,
 \item If there exist infinitely many odd squarefree integers $d$ such that $a(g_i, d) \ne 0$, then there exist infinitely many odd squarefree integers $d$ such that $a(g_{i-1}, d) \ne 0$.
\end{enumerate}

We proceed inductively. Put $g_0 = f$ and let $$g_1(z) = \sum_{\substack {n > 0\\ n \text{ odd }}}a(g_0, n) e(nz).$$ By Lemma~\ref{lema4}, we know that $g_1 \in S_{k + \frac12}(4N, \chi).$ We next show that $g_1 \neq 0$. Indeed, if $g_1 = 0$, then $a(g_0, n) = 0$ unless $n$ is even.  By Lemma~\ref{lema7}, this implies that $N/4$ is even, a contradiction. It is clear that $g_1$ satisfies the four listed properties.

 Next, suppose we have constructed $g_{i-1}$ with the listed properties. If $$\sum_{\substack {n > 0\\ (n, p_i) = 1}}a(g_{i-1}, n) e(nz) \ne 0 ,$$ then we put $$g_i(z) = \sum_{\substack {n > 0\\ (n, p_i) = 1}}a(g_{i-1}, n) e(nz).$$ By Lemma~\ref{lema4}, we know that $g_i \in S_{k + \frac12}(NN_{i-1}p_i^2, \chi \chi_{i-1})$ and clearly satisfies the listed properties.

 On the other hand if $\sum_{\substack {n > 0\\ (n, p_i) = 1}}a(g_{i-1}, n) e(nz) = 0 ,$ then $a(g_{i-1}, n) = 0$ for all $n$ not divisible by $p_i$. By Lemma~\ref{lema7}, there exists $$g_{i-1}' \in S_{k + \frac12}(NN_{i-1}/p_i, \chi \chi_{i-1} \epsilon_{p_i})$$ such that $g_{i-1}(z) = g_{i-1}'(p_iz)$; here $\epsilon_{p_i}$ is the quadratic character associated to the field $\Q(\sqrt{p_i}).$ Clearly $g_{i-1}' \ne 0$ because $g_{i-1} \neq 0$. In this case, we put $$g_i(z) = \sum_{\substack {n > 0\\ (n, p_i) = 1}}a(g_{i-1}', n) e(nz).$$

 By Lemma~\ref{lema4}, $g_i \in S_{k + \frac12}(NN_{i-1}p_i, \chi \chi_{i-1}\epsilon_{p_i}).$ We claim that $g_i \ne 0$. Indeed,  if $g_i = 0$, then $a(g_{i-1}', n) = 0$ for all $n$ not divisible by $p_i$. Again, by Lemma~\ref{lema7}, this implies that $p_i^2$ divides $N$ and that the conductor of $\chi \chi_{i-1}$ divides $NN_{i-1}/p_i^2.$ But this is impossible because $p_i$ divides the said conductor but not the quantity $NN_{i-1}/p_i^2.$

 Thus we have $g_i \neq 0$. Clearly $g_i$ satisfies the listed properties. So, inductively we have constructed the sequence of forms $g_i$, $1 \le i \le t$ with the listed properties.

 Finally put $f'= g_t$, $N' = NN_t$, $\chi' = \chi \chi_t$. Then $f' \in S_{k + \frac12}(N', \chi')$ satisfies the conditions of Proposition~\ref{keyprop}. So there are infinitely many odd squarefree integers $d$ such that $a(f',d) \neq 0$. By the last listed condition for $g_t$, the same is true of $g_0=f$. The proof of Theorem~\ref{impprop} is complete.

\bibliography{lfunction}

\end{document}